\documentclass[11pt]{amsart}
\usepackage{amssymb}
\usepackage{amscd}
\usepackage{mathrsfs}
\usepackage{cases}
\usepackage{latexsym}
\usepackage{amsmath}
\usepackage{indentfirst}
\usepackage{stmaryrd}
\usepackage{amsfonts, color}
\usepackage{amsmath,amssymb,amscd,bbm,amsthm,mathrsfs,dsfont}
\usepackage{graphicx}
\usepackage{fancyhdr}
\usepackage{amsxtra,ifthen}
\usepackage{verbatim}
\usepackage{latexsym,epsfig,amsmath}
\usepackage{hyperref}
\usepackage{cite}
\newtheorem{thm}{Theorem}[section]
\newtheorem{defi}{Definition}[section]
\newtheorem{lem}{Lemma}[section]

\newtheorem{cor}{Corollary}[section]
\newtheorem{question}{Question}[section]
\newcommand\lemref[1]{Lemma~\ref{#1}}
\newcommand\thmref[1]{Theorem~\ref{#1}}

\begin{document}
\title[]{Isomorphism classes and stably isomorphisms of  double  Danielewski varieties}
\author{Xiaosong Sun, Shuai Zeng}
\address{Xiaosong Sun: School of Mathematical Sciences, Jilin University,
	Changchun, Jilin, 130012,China}
\email{sunxs@jlu.edu.cn}
\address{Shuai Zeng: School of Mathematical Sciences, Jilin University,
	Changchun, Jilin, 130012,China}
\email{zengshuai22@jlu.edu.cn}

\thanks{This work was supported by the NSFC (12371020, 12171194), NSFJP
 (20220101019JC, 20210101469JC)}
\maketitle

\begin{abstract}
	The interest in Danielewski varieties arose from the study of the Cancellation Problem. In this paper, we study the isomorphism classes and stably isomorphisms of double Danielewski varieties, and show that they are counterexamples of the Cancellation Problem.
\end{abstract}

\section{Introduction}	

The Cancellation Problem is a central problem in affine algebraic geometry\cite{EakinHeinzer1973}. For a ring $R$, we denote by $R^{[n]}$ the polynomial ring over $R$ in $n$ variables.

\begin{question}(Cancellation Problem)
	Let $k$ be a field. If $F$ and $G$ are two finitely generated $k$-algebras such that $F^{[1]}\cong_{k}G^{[1]}$, then does it follow that  $F\cong_{k} G$?
\end{question}

In 1989, Danielewski \cite{danielewski1989cancellation} constructed a family of two-dimensional affine domains over the complex number field $\mathbb{C}$ which are counterexamples of the above question. For a non-constant polynomial $P(z)$ over $k=\mathbb{C}$ with distinct roots, Danielewski considered the  affine surfaces $D_{n}$ defined by $x^{n}y-P(z)=0$ in $\mathbb{A}_{k}^{3}$. Write $A(D_{n})$ for the coordinate ring of $D_{n}$. Now it is known that for any pair $(m,n)$ with $m\neq n,$ one has $D_{m}\ncong_{k} D_{n} $ but $D_{m} \times \mathbb{A}^{1}_{k}\cong_{k} D_{n}\times \mathbb{A}^{1}_{k}$, equivalently, $A(D_{m})\ncong_{k} A(D_{n})$ but $A(D_{m})^{[1]}\cong_{k} A(D_{n})^{[1]}$ (cf. \cite{fieseler1994complex}.)

In 2001, Makar-Limanov\cite{Makar2001} determined the automorphism group of the Danielewski surface $D_{n}$ by computing the Makar-Limanov invariant of $D_{n}$. When the field $k$ is of characteristic zero, the Makar-Limanov invariant $\textup{ML}(X)$ of an affine $k$-variety $X$ is the intersection of all the kernels of locally nilpotent derivations of the coordinate ring $A(X)$.

In 2019, Gupta and Sen \cite{gupta2019double} computed the Makar-Limanov invariant and isomorphism classes of double Danielewski surfaces defined  by equations of the form $\{x^{d}y-P(x,z)=0, x^{e}t-Q(x,y,z)=0\}$ in $\mathbb{A}_{k}^{4}$, where $k $ is a field of any characteristic, $d,e\in \mathbb{N}, P(x,z)$ is monic in $z$ and $Q(x, y,z)$ is monic in $y$ with $\deg_{z}P(x,z)\geq 2, \deg_{y}Q(x,y,z)\geq 2$ and then verified that these are counterexamples to the Cancellation Problem.

In 2007, Dubouloz \cite{dubouloz2007additive} considered a higher-dimensional analogue of Danie\-lewski surfaces. He constructed the Danielewski varieties defined by \[x_{1}^{d_{1}}\cdots x_{n}^{d_{n}}y-P(x_{1},\ldots,x_{n},z)=0\] in $\mathbb{A}_{\mathbb{C}}^{n+2}$, where $P(x_{1},\ldots, x_{n},z)$ is monic in $z$ and $\deg_{z}P>1$. He proved that they are counterexamples to the Cancellation Problem. In 2023,  Ghosh and Gupta \cite{GHOSH2023226} extended the result of Dubouloz \cite{dubouloz2007additive} to arbitrary characteristic.
	
In \cite{SunZeng2023}, we studied a higher-dimensional analogue of the double Danielewski surfaces. More precisely, we considered a family of affine varieties $L_{[d],[e]}$ of arbitrary dimension greater or equal to 2 over a field $k$, defined by a pair of equations
\[\{\underline{x}^{[d]}y-P(\underline{x},z)=0, \underline{x}^{[e]}t-Q(\underline{x},y,z)=0\} \ \text{in}\  \mathbb{A}^{n+3}_{k},\] where $[d],[e]\in \mathbb{N}^{n}_{\geq 1}, \underline{x}:=(x_{1},\ldots, x_{n}), P(\underline{x}, z)$ is monic in $z$ and $Q(\underline{x},y,z)$ is monic in $y$ with $\deg_{z}P(\underline{x},z)\geq 2,\deg_{y}Q(\underline{x},y,z)\geq 2$. We call them $\mathit{double\ Danielewski \ varieties}$.
We described the Makar-Limanov invariant and locally nilpotent derivations of these varieties in that paper.

In this paper, we will investigate the double Danielewski varieties further. In Section 2, we study the isomorphism classes of these varieties, and then we study in Section 3 the stably isomorphisms of these varieties, and obtain some counterexamples of the Cancellation Problem.
	
\section{Isomorphism classes}	
In this section we investigate the isomorphism classes of a family of varieties including the double Danielewski varieties.

We consider two such varieties, denoted by $L$ and $L'$,
\[L=\frac{k[X,Y,Z,T]}{(\underline{X}^{[d]}Y,-P(\underline{X},Z), \underline{X}^{[e]}T-Q(\underline{X},Y,Z)),}\]
\[L'=\frac{k[X,Y,Z,T]}{(\underline{X}^{[d']}Y,-P'(\underline{X},Z), \underline{X}^{[e']}T-Q'(\underline{X},Y,Z) )}\]
where $[d],[d'],[e],[e']\in \mathbb{N}_{\geq 1}^{n}$, $\underline{X}:= (X_{1},\ldots, X_{n})$, $P,P'\in k[\underline{X},Z]$ are monic polynomials in $Z$, and $Q,Q'\in k[\underline{X},Y,Z]$ are monic polynomials in $Y$, with $r=\deg_{Z}P(\underline{X},Z),r'=\deg_{Z}P'(\underline{X},Z)$ and $s=\deg_{Y}Q(\underline{X},Y,Z), s'=\deg_{Y}Q'(\underline{X},Y,Z)$. And we suppose the following conditions are satisfied.
 \begin{align}\label{eq: conditions}
	either\ & r\geq 2\  and \ s\geq 2 \\ \notag
	or\  & r\geq 2 \ and\  s=1\\   \notag
	or \ & r=1, s\geq 2 \ and \ [e]\in \mathbb{Z}_{\geq 2}^{n}.
\end{align}
Note that if $r\geq 2$ and $s\geq 2$ (resp. $r'\geq 2$ and $s'\geq 2$) then $L$ (resp. $L'$) is a double Danielewski variety.

 Let $\underline{x}, y, z,t$ and $\underline{x}', y',z',t'$ denote the images of $\underline{X},Y,Z,T$ in $L$ and $L'$ respectively.  By \cite[Theorem 3.10]{SunZeng2023}, $\mathrm{ML}(L)=k[\underline{x}], \mathrm{ML}(L')=k[\underline{x}'].$  Let $\sigma\in S_{n}$ be a permutation of order $n$.
Write $\sigma(x_{1},\ldots,x_{n}):=(x_{\sigma(1)},\ldots,x_{\sigma(n)}),$
$\sigma(d_{1},\ldots, d_{n}):=(d_{\sigma(1)},\ldots, d_{\sigma(n)})$ for the $\sigma$-actions, and for $\lambda=(\lambda_1,\ldots,\lambda_n)\in k^n$, we set $\lambda(\sigma(\underline{x})):=(\lambda_{1}x_{\sigma(1)},$
$\ldots, \lambda_{n}x_{\sigma(n)})$. The notation
$[e]\pm 1$ means to plus or minus $1$ in components of $[e]$, and $\underline{X}_{\hat{X_{i}}}$ means $\underline{X}\backslash\{X_{i}\}$.

\begin{thm}\label{thm: isomorphismthm}
$L\cong L'$ if and only if the following conditions hold:
	
	$\mathrm{(1)}$ $\{d_{1},\ldots,d_{n}\}=\{d'_{1},\ldots,d'_{n}\}$;
	
    $\mathrm{(2)}$ $\{e_{1},\ldots,e_{n}\}=\{e_{1}',\ldots,e_{n}'\}$;
	
	$\mathrm{(3)}$ $r=r', s=s'$;
	
	$\mathrm{(4)}$ There exist $\lambda:=(\lambda_{1},\ldots,\lambda_{n})\in (k^{*})^n,$ $\gamma\in k^{*},$ $\sigma\in S_{n},$ $\varrho(\underline{X})\in k[\underline{X}],$ $f(\underline{X},Z)\in k[\underline{X},Z]$ and $h(\underline{X},Y,Z)\in k[\underline{X},Y,Z]$ such that
	\begin{enumerate}
	\item  $P'(\lambda(\sigma(\underline{X})),\gamma Z+\varrho(\underline{X}))=\tau P(\underline{X},Z)+\underline{X}^{[d]}f(\underline{X},Z)$, where $\tau=\gamma^{r}(\in k^{*})$. In particular, $P'(\underline{0},\gamma Z+\varrho(\underline{0}))=\tau P(\underline{0},Z)$.
	\item  $Q'(\lambda(\sigma(\underline{X})), \upsilon Y+g(\underline{X},Z), \gamma Z+\varrho(\underline{X}))=\kappa Q(\underline{X},Y,Z)+\underline{X}^{[e]}h(\underline{X},Y,Z)$, where $\upsilon=\lambda^{[-\sigma(d)]}\tau, \kappa=\upsilon^{s}$ and $g(\underline{X},Z)=\lambda^{[-\sigma(d)]}f(\underline{X},Z)$. In particular, $Q'(\underline{0},\upsilon Y+g(\underline{0},Z), \gamma Z+\varrho(\underline{0}))=\kappa Q(\underline{0}, Y,Z).$
	\end{enumerate}
Moreover, if $\psi: L'\rightarrow L$ is an isomorphism, then
\begin{align}
	\psi(x_{1}')&=\lambda_{1}x_{\sigma(1)}\notag\\
	\psi(x_{2}')&=\lambda_{2}x_{\sigma(2)}\notag\\
	&\vdots  \notag\\
	\psi(x_{n}')&=\lambda_{n}x_{\sigma(n)}\notag
\end{align}
\[\psi(z')=\gamma z +\varrho(\underline{x}),\]
\[\psi(y')=\upsilon y+g(\underline{x},z)\ \text{and}\ \psi(t')=\lambda^{[-\sigma([e])]}(\kappa t+h(\underline{x},y,z)).\]

\end{thm}	
	
\begin{proof}
	Let $\psi : L\rightarrow L'$ be a $k$-algebra isomorphism. Replacing $L$ by $\psi(L)$, we may assume that $L=L'=M$. By Theorem 3.10 of \cite{SunZeng2023}, $\mathrm{ML}(M)=k[\underline{x}]=k[\underline{x}']$ and hence
\begin{align}\label{eq: xcoordichange1}
		x_{1}'&=L_{1}+\mu_{1}\notag\\
		x_{2}'&=L_{2}+\mu_{2}\notag\\
		               &\vdots  \notag\\
		x_{n}'&=L_{n}+\mu_{n}\notag
\end{align}
where $ L_{1}, \ldots, L_{n}$ are linear forms on $x_{1},\ldots, x_{n}$, and $\mu_{i}\in k[\underline{x}],i\in \{1,\ldots,n\}$ have no linear terms and $k(\underline{x})[z]=k(\underline{x}')[z']$. Since $M\cap k(\underline{x})=k[\underline{x}]$, we have, $z'=\gamma z+\varrho$ 
for some $\gamma(\underline{x}), \varrho(\underline{x})\in k[\underline{x}]$. Using symmetry, we obtain that $\gamma(\underline{x})\in k^{*}$, i.e.,
\begin{equation}\label{eq:zcoordichange}
	z'=\gamma z+ \varrho(\underline{x})
\end{equation}
for some $\gamma\in k^{*}, \varrho(\underline{x})\in k[\underline{x}]$. Hence,
\begin{equation}\label{eq: k[x,z]equality}
	k[\underline{x}, z]=k[\underline{x}', z'].
\end{equation}
As $y'\in M\subseteq k[\underline{x}, \underline{x}^{-1},z]$, there exists $[s]=(s_{1},\ldots, s_{n})\in \mathbb{Z}_{\geq 0}^{n}$ such that $\underline{x}^{[s]}y'\in k[\underline{x}, z]$. Since \[P'(\underline{x}',z')=\underline{x}^{{'}^{[d']}}y'= (L_{1}+\mu_{1})^{d_{1}'}\cdots (L_{n}+ \mu_{n})^{d_{n}'}y',\]  we have $\underline{x}^{[s]}P'(\underline{x}',z')\in (L_{1}+\mu_{1})^{d_{1}'}\cdots (L_{n}+ \mu_{n})^{d_{n}'} k[\underline{x},z]$. As $P'(\underline{X}, Z)$ is monic in $Z$, we deduce that $\mu_{i}=0, i\in \{1,\ldots,n\}$, and
\begin{align}\label{eq:xcoordichange2}
	x_{1}'&=\lambda_{1}x_{\sigma(1)}\notag\\
	x_{2}'&=\lambda_{2}x_{\sigma(2)}\notag\\
	&\vdots \\
	x_{n}'&=\lambda_{n}x_{\sigma(n)}\notag
\end{align}
where $\sigma\in S_{n}, \lambda_{i}\in k^{*}$ for $i\in \left\{1,\ldots, n\right\}$.

Now we show that $\{d_{1},\ldots,d_{n}\}=\{d'_{1},\ldots,d'_{n}\}$. Denote the maximum element in $\{d_{1},\ldots,d_{n},d_{1}',\ldots,d_{n}'\}$ by $k$ and let $[K]=(k,\ldots,k)$. Using \eqref{eq: k[x,z]equality} and \eqref{eq:xcoordichange2}, we have $\underline{x}^{[K]}M\cap k[\underline{x},z]=\underline{x}^{{'}^{[K]}}M\cap k[\underline{x}',z']$, i.e.,
\[(\underline{x}^{[K]}, \underline{x}^{([K]-[d])}P(\underline{x},z))k[\underline{x},z]=(\underline{x}^{{'}^{[K]}}, \underline{x}^{{'}^{([K]-[d'])}}, P'(\underline{x}', z'))k[\underline{x}',z'].\]
Therefore  $x_{1}^{k-d_{1}}\cdots x_{n}^{k-d_{n}}P(\underline{x},z)=\underline{x}^{[K]}g_{1}(\underline{x}',z')+ x_{1}^{(k-d_{\sigma^{-1}(1)}')}\cdots x_{n}^{(k-d_{\sigma^{-1}(n)}')}$
$g_{2}(\underline{x}',z')$,
where $g_{1}(\underline{x}',z'),g_{2}(\underline{x}',z')\in k[\underline{x}',z']$. It follows that
\[k-d_{i}\leq k-d'_{\sigma^{-1}(i)}, i\in \{1,\ldots,n\}\] as $P(\underline{X},Z)$ is monic in $Z$. By symmetry, we have $\{d_{1},\ldots,d_{n}\}$ are equal to $\{d'_{\sigma^{-1}(1)},\ldots,d'_{\sigma^{-1}(n)}\}$. So $\{d_{1},\ldots,d_{n}\}=\{d'_{1},\ldots,d'_{n}\}$.

Then we have\[\underline{x}^{[d]}M\cap k[\underline{x},z]= \underline{x}^{{'}^{\sigma([d])}}M\cap k[\underline{x}',z'],\] i.e.,
\begin{equation}
	(\underline{x}^{[d]}, P(\underline{x}, z))k[\underline{x},z]= (\underline{x}^{{'}^{\sigma([d])}},P'(\underline{x}', z'))k[\underline{x}',z'].
\end{equation}
Thus $P'(\underline{x}', z')=\tau' P(\underline{x},z)+ \underline{x}^{[d]}f'$ for some $\tau', f'\in k[\underline{x},z]$. Since $P,P'$ are monic in $Z$, using \eqref{eq:zcoordichange} and \eqref{eq:xcoordichange2}, we see that
\[r(=\deg_{Z}P)=r'(=\deg_{Z}P'),\] and $\tau' \equiv \gamma^{r} \mod \underline{x}^{[d]}k[\underline{x},z]$. Let $\tau =\gamma^{r}(\in k^{*})$. Replacing $\tau'$ by $\tau$, we have
\begin{equation}\label{eq:Pchange}
	P'(\underline{x}',z')= \tau P(\underline{x},z) +\underline{x}^{[d]}f(\underline{x},z)
\end{equation}
for some $f\in k[\underline{x},z]$. In particular, using \eqref{eq:zcoordichange} and \eqref{eq:xcoordichange2}, and putting $\underline{x}=\underline{0}$ in \eqref{eq:Pchange}, we obtained that
\[P'(\underline{0}, \gamma z+\varrho(0))=\tau P(\underline{0},z).\]
Now we have,
\begin{equation}\label{eq: ycoordichange}
	y'=\frac{P'(\underline{x}',z')}{\underline{x}^{{'}^{[d]'}}}=\frac{\tau P(\underline{x},z)+\underline{x}^{[d]}f(\underline{x},z)}{\lambda^{[\sigma{(d)}]}\underline{x}^{[d]}}=\upsilon y+g(\underline{x},z),
\end{equation}
where $\upsilon=\lambda^{[-\sigma{(d)}]}\tau\in k^{*}$ and $g=\lambda^{[-\sigma{(d)}]}f\in k[\underline{x},z]$. Therefore,
\begin{equation}\label{eq:k[x,y,z]equality}
	k[\underline{x},y,z] =k[\underline{x}',y',z'].
\end{equation}

Now we show that $\{e_{1},\ldots,e_{n}\}=\{e_{1}',\ldots,e_{n}'\}$. Denote the maximum element in $\{e_{1},\ldots,e_{n},e_{1}',\ldots,e_{n}'\}$ by $u$ and let $U=(u,\ldots,u)$. Then  \[\underline{x}^{[U]}M\cap k[\underline{x},y,z]=\underline{x}^{{'}^{[U]}}M\cap k[\underline{x}',y',z'],\] i.e.,
\[(\underline{x}^{[U]}, \underline{x}^{([U]-[e])}Q(\underline{x},y,z))k[\underline{x},y,z]=(\underline{x}^{{'}^{[U]}},\underline{x}^{{'}^{([U]-[e'])}}Q'(\underline{x}',y',z'))k[\underline{x}',y',z'].\]
Therefore  \begin{align*}& x_{1}^{u-e_{1}}\cdots x_{n}^{u-e_{n}}Q(\underline{x},y,z)\\&=\underline{x}^{[U]}g_{1}(\underline{x}',y',z')+ x_{1}^{(u-e_{\sigma^{-1}(1)}')}
\cdots x_{n}^{(u-e_{\sigma^{-1}(n)}')}
g_{2}(\underline{x}',y',z'),\end{align*}
where $g_{1}(\underline{x}',y',z'),g_{2}(\underline{x}',y',z')\in k[\underline{x}',y',z']$. Hence
\[u-e_{i}\leq u-e'_{\sigma^{-1}(i)}, i\in \{1,\ldots,n\}\] as $Q(\underline{X},Y,Z)$ is monic in $Y$. By symmetry, we have $\{e_{1},\ldots,e_{n}\}$ are equal to $\{e'_{\sigma^{-1}(1)},\ldots,e'_{\sigma^{-1}(n)}\}$. So $\{e_{1},\ldots,e_{n}\}=\{e'_{1},\ldots,e'_{n}\}$.

Now we have\[\underline{x}^{[e]}M\cap k[\underline{x},z]= \underline{x}^{{'}^{\sigma([e])}}M\cap k[\underline{x}',z'],\] i.e.,
\begin{equation}
	(\underline{x}^{[e]}, Q(\underline{x},y, z))k[\underline{x},y,z]= (\underline{x}^{{'}^{\sigma([e])}},Q'(\underline{x}', y',z'))k[\underline{x}',y',z'].
\end{equation}
Thus $Q'(\underline{x}', y',z')=\kappa' Q(\underline{x},y,z)+ \underline{x}^{[e]}h'$ for some $\kappa', h'\in k[\underline{x},y,z]$. Since $Q,Q'$ are monic in $Y$, using \eqref{eq:zcoordichange} and \eqref{eq: ycoordichange}, we get
\[s(=\deg_{Y}Q)=s'(=\deg_{Y}Q'),\] and $\kappa' \equiv \upsilon^{s} \mod \underline{x}^{[e]}k[\underline{x},y,z]$. Let $\kappa =\upsilon^{s}(\in k^{*})$. Replacing $\kappa'$ by $\kappa$, we have
\begin{equation}\label{eq:Qchange}
	Q'(\underline{x}',y',z')= \kappa Q(\underline{x},y,z) +\underline{x}^{[e]}h(\underline{x},y,z)
\end{equation}
for some $h\in k[\underline{x},y,z]$. In particular, using  \eqref{eq:zcoordichange},\eqref{eq:xcoordichange2} and \eqref{eq: ycoordichange}, and putting $\underline{x}=\underline{0}$ in \eqref{eq:Qchange}, we obtain that
\[Q'(\underline{0}, \upsilon y+g(\underline{0},z),\gamma z+\varrho(0))=\kappa Q(\underline{0},y,z).\]
Now we get
\begin{equation}\label{eq: tcoordichange}
	t'=\frac{Q'(\underline{x}',y',z')}{\underline{x}^{{'}^{[e]'}}}=\frac{\kappa Q(\underline{x},y,z)+\underline{x}^{[e]}h(\underline{x},y,z)}{\lambda^{[\sigma{(e)}]}\underline{x}^{[e]}}=\frac{\kappa t+h}{\lambda^{[\sigma{(e)}]}}.
\end{equation}
Conversely, suppose that the conditions $(1),(2),(3)$ and $(4)$ hold. Consider the $k$-algebra map $\xi: k[\underline{X},Y,Z,T]\rightarrow L$ defined by
\begin{align}
	\xi(X_{1})&=\lambda_{1}x_{\sigma(1)}\notag\\
	\xi(X_{2})&=\lambda_{2}x_{\sigma(2)}\notag\\
	&\vdots  \notag\\
	\xi(X_{n})&=\lambda_{n}x_{\sigma(n)}\notag\\
	\xi(Z)&=\gamma z +\varrho(\underline{x}),\notag\\
	\xi(Y)&=\upsilon y+g(\underline{x},z),\notag\\ \xi(T)&=\theta t+h'(\underline{x},y,z).\notag
\end{align}
where $\upsilon= \lambda^{[-\sigma([d])]}\gamma^{r},$ $~\theta=\lambda^{[-\sigma([e])]}\upsilon^{s},$ $~g(\underline{x},z)=\lambda^{[-\sigma([d])]}f(\underline{x},z)$ and $h'(\underline{x},y,z)=\lambda^{[-\sigma([e])]}h(\underline{x},y,z)$. Then clearly, \[\xi(\underline{X}^{[d]}Y-P'(\underline{X},Z))=\xi(\underline{X}^{[e]}T-Q'(\underline{X},Y,Z))=0.\] Thus $\xi$ induces a $k$-linear map $\overline{\xi}:L'\rightarrow L$, which is surjective. Since $L$ and $ L'$ are of same Krull dimension, we have $\overline{\xi}$ is an isomorphism.

\end{proof}	
As a byproduct, we can deduce the following result.
\begin{thm}
	Let $L$ be as in \thmref{thm: isomorphismthm} and let the parameters $r,s$ and $[e]$ satisfy  the condition \eqref{eq: conditions}. Denote by $S$  the subring $k[\underline{x},y,z]$ of $L$ and let $\zeta\in \mathrm{Aut}_{k}(L)$. Then:
	\begin{enumerate}
		\item  $\zeta(k[\underline{x},z])=k[\underline{x},z]$;
		\item  $\zeta(x_{i})=\lambda_{i}x_{\sigma(i)}$ for some $\sigma\in S_{n}, \lambda_{i}\in k^{*},  i\in \{1,\ldots,n\}$;
		\item  $\zeta((\underline{x}^{[d]},P(\underline{x},z))k[\underline{x},z])=(\underline{x}^{[d]},P(\underline{x},z))k[\underline{x},z]$;
		\item  $\zeta(k[\underline{x},y,z])=k[\underline{x},y,z]$;
		\item  $\zeta((\underline{x}^{[e]}, Q(\underline{x},y,z))S)=(\underline{x}^{[e]},Q(\underline{x},y,z))S$;
		\item $\zeta(t)=at +b$
		, where $a\in k^{*} $ and $b\in S$.
	\end{enumerate}
\begin{proof}
	It follows from the proof of \thmref{thm: isomorphismthm}.
\end{proof}
\end{thm}	
\section{Stably isomorphisms}

In this section, we study the stably isomorphisms of double Danielewski varieties. A stably isomorphism between two double Danielewski varieties $L$ and $L'$ means an isomorphism between  $L^{[1]}$ and $L'^{[1]}$. We will find stably isomorphisms between some non-isomorphic double Danielewski varieties, which provide counterexamples to the Cancellation Problem.

First we recall the definition and basic facts of exponential maps.

\begin{defi}
	Let $R$ be a $k$-algebra and $\phi : R\rightarrow R^{[1]}$  a $k$-algebra homomorphism. Denote $W$ an indeterminate over $R$ and write  $\phi : R\rightarrow R[W]$. Then $\phi$ is called an $exponential  \ map\ on \ R$ if the following properties are satisfied:
	\begin{enumerate}
		\item  $\epsilon\phi_{W}$ is identity on $R$, where $\epsilon: R[W]\rightarrow R$ is evaluating $W$ to 0.
		\item  $\phi_{T}\phi_{W}=\phi_{W+T}$, where $\phi_{T}: R\rightarrow R[V]$ is extended to a homomorphism $\phi_{T}: R[W]\rightarrow R[T,W]$ by setting $\phi_{T}(W)=W$.
	\end{enumerate}
	
\end{defi}
The subalgebra $R^{\phi}=\{a\in R\ |\ \phi(a)=a\}$ of $R$ is called the ring of invariants of $\phi$.

\begin{lem}\cite{gupta2014zariski}\label{lem:factorialproperty}
	Let $R$ be an affine domain over a field $k$. Suppose that there exists a non-trivial exponential map $\phi$ on $R$. Then, $R^{\phi}$ is factorially closed in $R$, i.e., given $f,g\in R\backslash \{0\}$ the condition $fg\in R^{\phi} $ implies $f\in R^{\phi}$ and $g\in R^{\phi}$.
	
\end{lem}

The following theorem presents some stably isomorphisms between non-isomorphic double Danielewski varieties.

\begin{thm}\label{thm: stablyisomorphism} Consider a double Danielewski variety
	\[L_{[d],[e]}=\frac{k[\underline{X},Y,Z,T]}{(\underline{X}^{[d]}Y-P(\underline{X},Z), \underline{X}^{[e]}T-Q(\underline{X},Y,Z))},\] where $[d],[e]\in \mathbb{N}^{n}_{\geq 1}, P(\underline{X},Z)$ is a monic polynomial in $Z$ with $r:=\deg_{Z}P(\underline{X},Z)$ $\geq 2$ and $Q(\underline{X},Y,Z)$ is a monic polynomial in $Y$ with $s:=\deg_{Y}Q(\underline{X},Y,Z)\geq 2$.
Write $P'(\underline{X},Z):=\frac{\partial (P(\underline{X},Z))}{\partial Z}$ and $Q'(\underline{X},Y,Z))):=\frac{\partial (Q(\underline{X},Y,Z))}{\partial Z}$.

Suppose that for all $i\in \{1,\ldots,n\}$,
	\[(P(\underline{X},Z)|_{X_{i}=0},P'(\underline{X},Z)|_{X_{i}=0})k[\underline{X}_{\hat{X_{i}}},Z]=k[\underline{X}_{\hat{X_{i}}},Z] \ \text{and}\ \]\[ (P(\underline{X},Z)|_{X_{i}=0},Q(\underline{X},Y,Z)|_{X_{i}=0},Q'(\underline{X},Y,Z)|_{X_{i}=0})k[\underline{X}_{\hat{X_{i}}},Y,Z]=k[\underline{X}_{\hat{X_{i}}},Y,Z].\] Then, for any $[e]\in \mathbb{N}_{\geq 2}^{n}$,
	\[L_{[d],[e]}^{[1]}\cong L_{[d],[e]-1}^{[1]}.\]
\end{thm}

\begin{proof}
	We use $L$ to represent $L_{[d],[e]}$.  Consider the exponential map $\xi : L\rightarrow L[V]$  on $L$ defined   by
	\begin{align}
		\xi(x_{i})&=x_{i},\  \text{for}\  i\in \{1,\ldots,n\},\notag\\
		\xi(z)&=z+\underline{x}^{[d]+[e]}V,\notag\\
		\xi(y)&=\frac{P(\underline{x},z+\underline{x}^{[d]+[e]}V)}{\underline{x}^{[d]}}=y+\underline{x}^{[e]}V\alpha(\underline{x},z,V),\notag\\
		\xi(t)&=\frac{Q(\underline{x},y+\underline{x}^{[e]}V\alpha(\underline{x},z,V), z+\underline{x}^{[e]+[d]}V)}{\underline{x}^{[e]}}=t+V\beta(\underline{x},y,z,V),\notag
	\end{align}
where $\alpha(\underline{x},z,V)\in k[\underline{x},z,V]$ and $\beta(\underline{x},y,z,V)\in k[\underline{x},y,z,V]$. Let $C=L[w]=L^{[1]}$ and extend $\xi$ to $C$ by defining $\xi(w)=w-x_{i}V$. Let $r_{1}=\underline{x}^{[d]+[e]-1}w+z$. Then $r_{1}\in C^{\xi}$. Now,
\[P(\underline{x},r_{1})-P(\underline{x},z)=\underline{x}^{[d]+[e]-1}(P'(\underline{x},z)|_{x_{i}=0}w+x_{i}\theta)\] for some $\theta \in C$. Hence
\begin{equation}\label{eq: P(x,f)expansion}
	\begin{split}
		P(\underline{x},r_{1})&=P(\underline{x},z)+\underline{x}^{[d]+[e]-1}(P'(\underline{x},z)|_{x_{i}=0}w+x_{i}\theta)\\
		&=\underline{x}^{[d]}y+\underline{x}^{[d]+[e]-1}((P'(\underline{x}|,z)|_{x_{i}=0}w+x_{i}\theta)\\
		&=\underline{x}^{[d]}r_{2},
	\end{split}
\end{equation}
where $r_{2}=y+\underline{x}^{[e]-1}((P'(\underline{x},z)|_{x_{i}=0}w+x_{i}\theta)\in C$. Since $\underline{x}^{[d]}r_{2}=P(\underline{x},r_{1})\in C^{\xi}$ and $C^{\xi}$ is factorially closed in $C$, we have $r_{2}\in C^{\xi}$. Then
\begin{align}
	Q(\underline{x},r_{2},r_{1})&=Q(\underline{x}, y+\underline{x}^{[e]-1}((P'(\underline{x},z)|_{x_{i}=0}w+x_{i}\theta),z+\underline{x}^{[d]+[e]-1}w)\notag\\
	&=Q(\underline{x},y,z)+\underline{x}^{[e]-1}P'(\underline{x},z)|_{x_{i}=0}Q'(\underline{x},y,z)|_{x_{i}=0}w+\underline{x}^{[e]}\rho\notag
\end{align}
for some $\rho\in C$. Therefore,
\begin{equation}\label{eq: Q(x,g,f)expansion}
	\begin{split}
		Q(\underline{x},r_{2},r_{1})&=\underline{x}^{[e]}t+(Q(\underline{x},r_{2},r_{1})-Q(\underline{x},y,z))\\
		&=\underline{x}^{[e]}t+\underline{x}^{[e]-1}(P'(\underline{x},z)|_{x_{i}=0}Q'(\underline{x},y,z)|_{x_{i}=0}w+x_{i}\rho)\\
		&=\underline{x}^{[e]-1}(P'(\underline{x},z)|_{x_{i}=0}Q'(\underline{x},y,z)|_{x_{i}=0}w+x_{i}t+x_{i}\rho)\\
		&=\underline{x}^{[e]-1}r_{3}
	\end{split}
\end{equation}
where $r_{3}=P'(\underline{x},z)|_{x_{i}=0}Q'(\underline{x},y,z)|_{x_{i}=0}w+x_{i}t+x_{i}\rho\in C$.

As $Q(\underline{x},r_{2},r_{1})=\underline{x}^{[e]-1}r_{3}\in C^{\xi}$ and $C^{\xi }$ is factorially closed in $C$ by \lemref{lem:factorialproperty}, we have $r_{3}\in C^{\xi}$. Then we obtain that, in $k[\underline{X}_{\hat{X_{i}}},Y,Z]$, \begin{align*}(P(\underline{X},Z)|_{X_{i}=0},Q(\underline{X},Y,Z)|_{X_{i}=0},
Q'(\underline{X},Y,Z)|_{X_{i}=0}
P'(\underline{X},Z)|_{X_{i}=0})=(1).\end{align*} So there exist $a(Y,Z), b(Y,Z),$
$c(Y,Z)\in k[\underline{X}_{\hat{X_{i}}},Y,Z]$ such that
\[Q'(\underline{X},Y,Z)|_{X_{i}=0}P'(\underline{X},Z)|_{X_{i}=0}a(\underline{X}_{\hat{X_{i}}},Y,Z)+\]\[Q(\underline{X},Y,Z)|_{X_{i}=0}b(\underline{X}_{\hat{X_{i}}},Y,Z)+P(\underline{X},Z)|_{X_{i}=0}c(\underline{X}_{\hat{X_{i}}},Y,Z)=1.\] Since $Q(\underline{x},y,z)|_{x_{i}=0}\in x_{i}C$ and $P(\underline{x},z)|_{x_{i}=0}\in x_{i}C$, we have,
\begin{equation}\label{eq: relativeprimeeq}
	Q'(\underline{x},y,z)|_{x_{i}=0}P'(\underline{x},z)|_{x_{i}=0}a(\underline{X}_{\hat{X_{i}}},y,z)+x_{i}\varrho=1
\end{equation}
for some $\varrho\in C$.

Let $s=\frac{w-a(r_{2},r_{1})r_{3}}{x_{i}}$. We first show that $s\in C$. Note that $a(r_{2},r_{1})-a(y,z)\in x_{i}C$. Let $a(r_{2},r_{1})-a(y,z)=x_{i}\gamma$ for some $\gamma \in C$. Then
\begin{align}
	w-r_{3}a(r_{2},r_{1})&=w-r_{3}a(y,z)-r_{3}(a(r_{2},r_{1})-a(y,z))\notag\\
	&=w-r_{3}a(y,z)-r_{3}x_{i}\gamma\notag\\
	&=w-a(y,z)(P'(\underline{x},z)|_{x_{i}=0}Q'(\underline{x},y,z)_{x_{i}=0}w+x_{i}t+x_{i}\rho)-r_{3}x_{i}\gamma\notag\\
	&=w(1-a(y,z)P'(\underline{x},z)|_{x_{i}=0}Q'(\underline{x},y,z)|_{x_{i}=0})\notag\\
	 &\quad\quad-x_{i}(a(y,z)t+a(y,z)\rho+r_{3}\gamma)\notag\\
	&=wx_{i}\varrho-x_{i}(a(y,z)t+a(y,z)\rho+r_{3}\gamma)\in x_{i}C\notag
\end{align}
Thus, $s\in C$ and $\xi(s)=s-V$. Therefore,
\begin{equation}\label{eq: Csliceexpansion}
	C=C^{\xi}[v]=(C^{\xi})^{[1]}.
\end{equation}
Let $O=k[x,r_{1},r_{2},r_{3}]$. Consider indeterminates $X, R_{1},R_{2}$ and $R_{3}$ over $k$ so that $k[X,R_{1},R_{2},R_{3}]=k^{[4]}$ and let
\begin{equation}\label{eq: O1isomortominusonecase}
	O_{1}=\frac{k[X,R_{1},R_{2},R_{3}]}{\underline{X}^{[d]}R_{2}-P(\underline{X},R_{1}),X^{[e]-1}R_{3}-Q(\underline{X},R_{2},R_{1})}\cong L_{[d],[e]-1}
\end{equation}

We will show that $O\cong O_{1}$. Define a $k$-algebra homomorphism $\Xi:k[X,R_{1},R_{2},R_{3}]\rightarrow O$ such that \[\Xi(X)=x, \Xi(R_{1})=r_{1},\Xi(R_{2})=r_{2},\Xi(R_{3})=r_{3}.\] We can deduce from \eqref{eq: P(x,f)expansion} and \eqref{eq: Q(x,g,f)expansion} that $\Xi$ is surjective, and thus $\Xi$ induces a $k$-algebra surjective homomorphism $\overline{\Xi}:O_{1}\rightarrow O$. Since $C[1/x_{i}]=O[1/x_{i}][w]=O[1/x_{i}]^{[1]}$, $O$ has dimension $n+1$. Therefore, as $O_{1}$ is an integral domain (Lemma 3.4 of \cite{SunZeng2023}) of dimension $n+1$, $\overline{\Xi}$ is an isomorphism.

Now, we can show that $C^{\xi}=O$. The inclusion $O\subseteq C^{\xi}$ is obvious. And since $C[1/x_{i}]=O[1/x_{i}][w]$, it follows that $O[1/x_{i}]=C^{\xi}[1/x_{i}]$. Therefore, to show that $O=C^{\xi}$, it suffices  to show that $x_{i}O=x_{i}C^{\xi}\cap O$. Since $x_{i}C^{\xi}=x_{i}C\cap C^{\xi}$ by the factorial closed property of $C^{\xi}$, it is enough to show that $x_{i}O=x_{i}C\cap O$, equivalently to show that the kernel of the map $\nu: O\rightarrow C/x_{i}C$ is $x_{i}O$. We let $\tilde{u}$ denote the image of $u$ in $C/x_{i}C$, for all $u\in C$. We see that
\begin{align}
	C/x_{i}C&=\frac{k[\underline{X}_{\hat{X_{i}}},Y,Z,T,W]}{(P(\underline{X},Z)|_{X_{i}=0}, Q(\underline{X},Y,Z)|_{X_{i}=0})}\notag\\
	&=\left(\frac{k[\underline{X}_{\hat{X_{i}}},Y,Z]}{(P(\underline{X},Z)|_{X_{i}=0}, Q(\underline{X},Y,Z)|_{X_{i}=0})}\right)[W,T]
	=k[\tilde{z},\tilde{y},\tilde{w},\tilde{t}].\notag
\end{align}
Clearly $\nu(r_{1})=\tilde{z},\nu(r_{2})=\tilde{y}$ and $\nu(r_{3})=P'(\tilde{\underline{x}},\tilde{z})|_{\tilde{x_{i}}=0}Q'(\tilde{\underline{x}},\tilde{y},\tilde{z})|_{\tilde{x_{i}}=0}\tilde{w}$. Now by \eqref{eq: relativeprimeeq}, $P'(\tilde{\underline{x}},\tilde{z})|_{\tilde{x_{i}}=0}Q'(\tilde{\underline{x}},\tilde{y},\tilde{z})|_{\tilde{x_{i}}=0}\in (C/x_{i}C)^{*}$. Hence, $\nu(O)=k[\tilde{\underline{x}}_{\hat{x_{i}}},\tilde{z},\tilde{y},\tilde{w}]$ and $C/x_{i}C=\nu(O)[\tilde{t}]=\nu(O)^{[i]}$. Thus $\dim\nu(O)=\dim(C/x_{i}C)-1=n$. Now as
\begin{align}
	O/x_{i}O\cong O_{1}/x_{i}O_{1}&=k[\underline{X}_{\hat{X_{i}}},R_{1},R_{2},R_{3}]/(P(\underline{X},R_{1})|_{X_{i}=0}, Q(\underline{X},R_{2},R_{1})|_{X_{i}=0})\notag\\
	&=\left(\frac{k[\underline{X}_{\hat{X_{i}}},R_{1},R_{2}]}{P(\underline{X},R_{1})|_{X_{i}=0}, Q(\underline{X},R_{2},R_{1})|_{X_{i}=0}}\right)^{[1]},\notag
\end{align}
we have $\dim(O/x_{i}O)=n=\dim\nu(O)$. Hence the kernel of $\nu$ is $x_{i}O$, and  thus $C^{\xi}=O$. Then by \eqref{eq: Csliceexpansion}, we obtain that $L_{[d],[e]}[w]=C=(C^{\xi})^{[1]}=O^{[1]}\cong O_{1}^{[1]}$, i.e., $L_{[d],[e]}^{[1]}\cong L_{[d],[e]-1}$

\end{proof}	

Finally, combining \thmref{thm: isomorphismthm} and \thmref{thm: stablyisomorphism}, we have
\begin{cor}
	Under the hypotheses of \thmref{thm: stablyisomorphism}, for every $[d],[e]\in \mathbb{N}^{n}_{\geq 1}$,
	\[L_{[d],[e]}\ncong L_{[d],[e]+1}\ \text{but}\  L_{[d],[e]}^{[1]}\cong L_{[d],[e]+1}^{[1]}.\] This provides counterexamples to the Cancellation Problem.
\end{cor}

\end{document}